\documentclass[11pt,a4paper,reqno]{amsart}
\usepackage{amssymb,amsmath,amsfonts,hyperref,amsthm,graphicx,color,float,
mathtools,bbm}

\usepackage{tikz}
\usepackage{todonotes}
\usepackage[normalem]{ulem}
\usepackage{fourier}
\usepackage{pgfplots}
\usepackage{subcaption}
\usepackage{stackrel}
\usepackage{nicefrac}

\numberwithin{figure}{section}
\numberwithin{equation}{section}

\newtheorem{theorem}{Theorem}
\newtheorem{definition}[theorem]{Definition}
\newtheorem{lemma}[theorem]{Lemma}
\newtheorem{proposition}[theorem]{Proposition}

\newtheorem{corollary}[theorem]{Corollary}

\newtheorem{remark}{Remark}
\newtheorem*{example}{Example}

\numberwithin{theorem}{section}

\newcommand{\N}     {\mathbb{N}}
\renewcommand{\P}   {\mathbb{P}}

\newcommand{\Z}     {\mathbb{Z}}
\newcommand{\Zd}    {{\mathbb{Z}^d}}

\newcommand{\sss}	{\scriptscriptstyle}

\usepgfplotslibrary{fillbetween}

\newcommand{\Cinf}{\mathcal{C}_\infty}
\makeatletter
\@namedef{subjclassname@2020}{\textup{2020} Mathematics Subject Classification}
\makeatother

\title[Graph distances in scale-free percolation: the logarithmic case]{Graph distances in scale-free percolation: \\ the logarithmic case}

\author{Nannan Hao}
\author{Markus Heydenreich}
\address{Mathematisches Institut, Ludwig Maximilians Universit\"at M\"unchen, Theresienstr.~39, 80333~Munich, Germany}
\email{hao@math.lmu.de, m.heydenreich@lmu.de}
\date{\today}

\begin{document}

\begin{abstract}
Scale-free percolation is a stochastic model for complex networks. In this spatial random graph model, vertices $x,y\in\mathbb{Z}^d$ are linked by an edge with probability depending on i.i.d.\ vertex weights and the Euclidean distance $|x-y|$. Depending on the various parameters involved, we get a rich phase diagram. We study graph distances and compare it to the Euclidean distance of the vertices. Our main attention is on a regime where graph distances are (poly-)logarithmic in the Euclidean distance. We obtain improved bounds on the logarithmic exponents. In the light tail regime, the correct exponent is identified.
\end{abstract}

\maketitle

\section{Introduction}\label{intro}
\subsection{The model}~
We study \textit{scale-free percolation}, which we henceforth abbreviate as SFP. 
This is a stochastic model for real networks such as social networks, biological networks, the internet, etc., which was introduced by Deijfen, van der Hofstad and Hooghiemstra in 2013 \cite{Hof13}. 
Many real networks share ubiquitous features such as scale-free degrees and small-world behaviour. Our model, SFP, is an infinite spatial random graph model that exhibits these features; it is embedded into the hypbercubic lattice $\mathbb{Z}^d$ and shows geometric clustering. Closely related models are based on point processes rather than the fixed grid structure of $\mathbb Z^d$, and such models have been studied on finite and infinite domains. We discuss these variants in Section \ref{sec-lit}. 

We now describe the model in detail. We consider the lattice $\mathbb{Z}^d$ with fixed dimension $d\ge1$ and construct a locally finite random subgraph of the complete graph on the vertex set $\mathbb{Z}^d$. Recall that a graph is called locally finite if all its vertices have finite degrees. To each vertex $x\in \mathbb{Z}^d$, we assign an i.i.d.\ weight $W_x$ which follows a power-law distribution with parameter $\tau-1$ ($\tau>1$), that is, 
\begin{align}\label{weight}
\mathbb{P}(W_x\geq w)= w^{-(\tau-1)},\quad w\geq 1.
\end{align}
Conditioning upon these weights, we declare an edge $\{x,y\}$ to be \emph{open} independently of the status of other edges with probability
\begin{align}\label{prob}
p_{x,y}=1-\exp\left(-\lambda\frac{W_xW_y}{|x-y|^{\alpha}}\right),
\end{align} 
where $|\,\cdot\,|$ denotes the Euclidean norm and $\alpha, \lambda>0$ are further parameters of the model. We write $x\sim y$ if the edge $\{x,y\}$ is open. 
The object of interest in the present study is the subgraph induced by the open edges, we call its connected components \emph{clusters}.

Scale-free percolation indeed generates scale-free networks in the sense that the degrees of vertices follow a power-law distribution with tail exponent
\begin{equation}\label{gamma}
\gamma=\alpha(\tau-1)/d.    
\end{equation}
That is,
\begin{equation}
\mathbb{P}(D_x\geq k)=k^{-\gamma}\ell(k),\qquad k\in\N,
\end{equation}
where $D_x$ is the degree of $x\in \mathbb{Z}^d$ and $\ell$ is slowly varying at infinity, cf.\ \cite{Hof13}.

The focus of the present paper is on graph distances. Recall that the graph distance between two vertices is defined as the length of a shortest open path connecting them. If the vertices are in different clusters (and hence such open paths do not exist), then the graph distance is $\infty$. 
Graph distances in real networks, in particularly social networks, have been in the focus of network research since Milgram's experimental discovery of the small world effect (casually phrased as ``six degrees of separation''), and have also been investigated theoretically since then, e.g. \cite{Kle99,Mil67}. 

On finite networks, say with $N$ vertices, ``small world'' means that the graph distance between two points is much shorter than a regular structure would suggest, e.g.\ $(\log N)^{O(1)}$ as $N\to\infty$. Our network is infinite, and we therefore give a different interpretation to the small-world effect. We call an infinite subgraph $\mathcal C\subset\Z^d$ a \emph{small-world} graph if the graph distance $D(x,y)$ on $\mathcal C$ is much smaller than the Euclidean distance, that is if
\begin{equation} \label{eqSmallWorld}
	D(x,y)=\big(\log|x-y|\big)^{O(1)}\qquad \text{as $|x-y|\to\infty$}.
\end{equation}

For graph distances in scale-free percolation, a rich phase diagram has been established in the literature: conditional on two points, $x$ and $y$ say, to be in the (unique) infinite component, we get that with high probability (as $|x-y|\to\infty$)
\begin{itemize}
\item if $\gamma\le1$ then $D(x,y)\le 2$, cf.\ \cite{Hey17};
\item if $\alpha<d$, then $D(x,y)\leq \lceil d/(d-\alpha) \rceil$, cf.\ \cite{Hey17};
\item if $\gamma \in (1,2)$ and $\alpha > d$, then $D(x,y)=\frac{2}{|\log (\gamma-1)|}\log\log|x-y|$, cf.\ \cite{Hof13,Hof17};
\item if $\gamma>2$ and $\alpha>2d$, then $D(x,y)\gtrsim |x-y|$, cf.\ \cite{Dep15,New86}.
\end{itemize}

This behaviour (together with our new results) is summarized in Figure \ref{overview}. 
The results in the first three cases are referred as ``ultra-small world'' phenomenon, because the asymptotics are of smaller order than the requirements of \eqref{eqSmallWorld}. 
In these regimes, shortest paths are typically formed by vertices that have the highest weight in a certain neighborhood (\emph{locally dominanting vertices} or \emph{hubs}). 
In contrast, for $d<\alpha < 2d$ and $\gamma >2$, the weights are more homogeneous, and it is not sufficient to consider only dominant vertices to find the shortest paths. In this regime, there is a fine interplay between weights and spatial position of various vertices, which leads to (poly-)logarithmic upper and lower bounds on graph distances. The aim of this paper is to identify the right logarithmic power, thereby completing the phase diagram.

\begin{figure}[!h]
\centering
\begin{tikzpicture}
\begin{axis}[axis lines=middle, xlabel = $\tau$, ylabel = $\alpha$,xmin=1, xmax=4.2,
    ymin=0, ymax=2.9,
    xtick={2,3,4}, ytick=\empty,x post scale=1.6]
\addplot [
    domain=1.1:4.2, 
    samples=100, 
    color=blue,
    ]
    {1/(x-1)};
\addplot [name path=A,
    domain=1.7:3, 
    samples=100, 
    color=red,
    ]
    {2/(x-1)};
\addplot [name path=C,
    domain=2.5:3, 
    samples=100, 
    color=red,
    ]
    {1/(x-2)};
\addplot [name path=D,
    domain=2:4.2, 
    samples=100, 
    color=black,
    ]
    {1};
\addplot [name path=B,
    domain=2:4.2, 
    samples=100, 
    color=black,
    ]
    {2};
\addplot[red] coordinates {(3, 1) (3, 2)};
\addplot[blue!10] fill between[of=C and B, soft clip={domain=2.5:3}];
\addplot[yellow!10] fill between[of=A and C, soft clip={domain=2.5:3}];
\addplot[yellow!10] fill between[of=A and B, soft clip={domain=2:2.5}];
\addplot[green!10] fill between[of=D and B, soft clip={domain=3:4.2}];

\end{axis}
\node[] at (10.5,2.2) 
    {\small $\alpha=d$};
    \node[] at (10.5,4.2) 
    {\small $\alpha=2d$};
    \node[] at (1.8,5.5) 
    {\small $\gamma=1$};
    \node[] at (3.1,5.5) 
    {\small $\gamma=2$};
    \node[] at (1.7,1.7) 
    {\small $D(x,y)\leq 2$};
    \node[] at (8.5,1.6) 
    {\small $D(x,y)\leq \lceil d/(d-\alpha) \rceil$};
    \node[] at (3.4,2.9) 
    {\small $D(x,y)\approx$};
    \node[] at (4.2,2.4) 
    {\small $ \log \log |x-y|$};
    \node[] at (6,5.2) 
    {\small $D(x,y) \gtrsim |x-y|$};
    \node[] at (8.9,2.7) 
    {\small $D(x,y) \approx(\log|x-y|)^{\Delta}$};
    \node[] at (5.4,4.15) 
    {\small $ \alpha(\tau-2)=d$};
        \node[] at (4.8,3.3) 
    { $(a)$};
        \node[] at (6.3,3.3) 
    { $(b)$};
        \node[] at (8.4,3.3) 
    { $(c)$};
\end{tikzpicture}
\caption{Graph distances in different regimes of scale-free percolation. The regions in shadow are those we are interested in. The areas (a), (b) and (c) represent our improved bounds established in Theorem \ref{thm21}.}\label{overview}
\end{figure}

At the phase boundaries ($\gamma=1$ and $\gamma=2$) we expect that the graph distances depend on the precise tail behaviour of the connectivity function in \eqref{prob}, so that any universality is lost.

\subsection{Main results}
Before stating the main results, we first introduce some parameters:
\begin{align}
	&\alpha_1:=\alpha \wedge \frac{\alpha(\tau-1)}{2}=\alpha\wedge\frac{\gamma d}2, 
	\quad \alpha_2 :=\alpha \wedge \left(\alpha(\tau-1)-d\right)\label{alpha}=\alpha\wedge (\gamma-1)d,\\
	&\Delta:=\frac{\log 2}{\log (2d/{\alpha})},\quad\Delta_1:=\frac{\log 2}{\log (2d/{\alpha_1})},\quad  	\Delta_2:=\frac{\log 2}{\log (2d/{\alpha_2})}\label{delta}.
\end{align}
Here $x\wedge y$ means the minimum of $x$ and $y$.
If $\gamma$ in \eqref{gamma} is larger than 2, then 
\begin{equation}
d<\alpha_1\leq\alpha_2\leq\alpha<2d.
\end{equation}
As a consequence
\begin{equation}
1<\Delta_1\leq\Delta_2\leq\Delta.
\end{equation}

Deijfen et al. showed in \cite{Hof13} that for $d<\alpha<2d$ and $\gamma>1$, if $\P(W=0)<1$, then there exists a critical value $\lambda_c\in(0,\infty)$ such that for $\lambda>\lambda_c$ there exists a unique infinite cluster. We thus may condition on two vertices $x$ and $y$ to be in the same infinite cluster.

\begin{theorem}\label{thm21}
For scale-free percolation with parameters $\lambda>\lambda_c,\gamma>2$, and $d<\alpha<2d$, we have that for any $\epsilon>0$,
\begin{equation}
\lim_{|x-y|\to \infty}\mathbb{P}\left(\left(\log |x-y|\right)^{\Delta_1-\epsilon}\leq D(x,y)\leq \left(\log |x-y|\right)^{\Delta_2+\epsilon}\Big| x,y\in\Cinf\right)=1.
\end{equation}
\end{theorem}

Depending on the value of $\gamma$ and $\alpha$, the various minima in \eqref{alpha} give rise to three different regimes. These are depicted in Figure \ref{overview}. Writing $\mathcal{C}_{\infty}$ for the unique infinite cluster in the graph, we get  

\begin{itemize}
\item[(a)] for $\gamma>2$, $\alpha(\tau-2)<d$ and arbitrary $\epsilon>0$,
\begin{equation}\nonumber
	\lim_{|x-y|\to \infty}\mathbb{P}\left(\left(\log |x-y|\right)^{\Delta_1-\epsilon}\leq D(x,y)\leq \left(\log |x-y|\right)^{\Delta_2+\epsilon}\Big| x,y\in\Cinf\right)=1;
\end{equation}
\item[(b)] for $\tau<3$, $\alpha(\tau-2)\geq d$ and arbitrary $\epsilon>0$, 
\begin{equation}\nonumber
	\lim_{|x-y|\to \infty}\mathbb{P}\left(\left(\log |x-y|\right)^{\Delta_1-\epsilon}\leq D(x,y)\leq \left(\log |x-y|\right)^{\Delta+\varepsilon}\Big| x,y\in\Cinf\right)=1;
\end{equation}
\item[(c)] for $\tau\geq 3$ and arbitrary $\epsilon>0$, 
\begin{equation}\nonumber
	\lim_{|x-y|\to \infty}\mathbb{P}\left(\left(\log |x-y|\right)^{\Delta-\epsilon}\leq D(x,y)\leq \left(\log |x-y|\right)^{\Delta+\varepsilon}\Big| x,y\in\Cinf\right)=1.
\end{equation}
\end{itemize}
Note that here the upper bounds in Part (b) and (c) are from \cite{Dep15}.

Despite the improvements in both the upper and lower bounds, the reader may observe that there is still a gap between them in case (a) and (b) in our result. Therefore, it remains open as to what the correct exponent is. The main difficulty in closing the gap between the upper and lower bounds is that we do not have a precise estimate for the probability of a path being open in scale-free percolation. Lemma \ref{path} gives a nice upper bound. However, in view of Proposition \ref{adjedge}, it appears that this bound is not optimal for $\tau<3$. As shown in Proposition \ref{adjedge}, the actual asymptotics of the probability of a path being open in SFP are heterogeneous in the exponents of edges, which poses a great difficulty.

In fact, our methods also apply to more general forms of connection probabilities than \eqref{prob}. We return to this observation in Remark \ref{remark}. If we make the extra assumption that \emph{additionally all nearest-neighbour edges are open}, then a comparison with long-range percolation (explained in the following paragraph) gives the following improvement to parts (b) and (c) above: there exists $C>0$ such that 
\begin{equation}\label{upperbound2}
	\lim_{|x-y|\to \infty}\mathbb{P}\left(D(x,y)\leq C\left(\log |x-y|\right)^{\Delta}\right)=1.
\end{equation}
Mind that the extra assumption ensures that $x,y\in\Cinf$.


\subsection{Comparison with long-range percolation}\label{lrpi}

Before we proceed to the proofs of the results, we first introduce a related (though easier) model named \emph{long-range percolation}. Our analysis of scale-free percolation is crucially based on techniques developed for long-range percolation.

Long-range percolation (henceforth LRP) is also defined on the lattice $\mathbb{Z}^d$ for fixed dimension $d\geq 1$. Independently of all the other edges, the edge $\{x,y\}$ is open with probability $p^{\sss \rm LRP}_{xy}$. A typical choice of $p^{\sss \rm LRP}_{xy}$ is
\begin{equation}
p^{\sss \rm LRP}_{xy}=1-\exp\left(-\frac{\lambda}{|x-y|^{\alpha}}\right).
\end{equation}
Note that $p^{\sss \rm LRP}_{xy}$ is equal to $p_{xy}$ for scale-free percolation (as defined in \eqref{prob}) if $W_x\equiv1$ or $\tau=\infty$. 

Biskup et al.\ studied the graph distances in long-range percolation and obtained sharp results.
\begin{theorem}[{Biskup \cite{Bis04}, Trapman \cite{Trp10}, Biskup-Lin \cite{Bis17}}]\label{lrp}
Consider the long-range percolation with connection probability $\{p_{xy}\}$ such that
\begin{equation}\label{asymlrp}
\liminf_{|x-y|\to \infty}p^{\sss \rm LRP}_{xy}|x-y|^{\alpha} >0,
\end{equation}
for some $\alpha>0$.
If $d <\alpha <2d$ and a unique infinite open cluster exists, then for all $\epsilon>0$ one has
\begin{equation} \label{bis04}
\lim_{|x-y|\to \infty}\mathbb{P}\left(\left( \log |x-y|\right)^{\Delta-\epsilon}\leq D(x,y) \leq \left(\log |x-y|\right)^{\Delta+\epsilon}\Big|x,y \in \Cinf\right)=1.
\end{equation}
If, moreover, we have the stronger form of connection probability
\begin{align}
p^{\sss \rm LRP}_{xy}=1-\exp\left(-\frac{\lambda}{|x-y|^{\alpha}}\right), 
\end{align}
and assume the existence of all nearest-neighbour edges, then there exist constants $C>c>0$ such that
\begin{equation} \label{upper}
\lim_{|x-y|\to \infty}\mathbb{P}\left(c\left( \log |x-y|\right)^{\Delta}\leq D(x,y) \leq C\left(\log |x-y|\right)^{\Delta}\right)=1.
\end{equation}
\end{theorem}

Trapman \cite{Trp10}, moreover, identified the growth of the balls $\{x\in\Zd:D(0,x)\le n\}$ for LRP with $d<\alpha<2d$.

Now we can describe a coupling between LRP and SFP. To this end, we view the two models from another perspective: to each edge $\{x,y\}$ of the graph, we assign an i.i.d.\ Uniform$[0,1]$-distributed random variable $U_{xy}$. Then, for scale-free percolation model, we consider for each edge $\{x,y\}$ the event 
\begin{align}
A_{x,y}:=\left\lbrace U_{xy}\leq 1-\exp\left(-\lambda\frac{W_xW_y}{|x-y|^{\alpha}}\right)\right\rbrace,
\end{align}
and we make the edge $\{x,y\}$ open whenever $A_{x,y}$ occurs. 
In the same way, for long-range percolation we consider the event
\begin{equation}
B_{x,y}:=\left\lbrace U_{xy}\leq 1-\exp\left(-\lambda\frac{1}{|x-y|^{\alpha}}\right)\right\rbrace.
\end{equation}
We have thus constructed a coupling for the two models: since $W_x\geq 1$ for all $x\in \mathbb{Z}^d$, we have 
\begin{align*}
1-\exp \left(-\lambda \frac{1}{|x-y|^{\alpha}}\right) \leq 1-\exp \left(-\lambda \frac{W_x W_y}{|x-y|^{\alpha}}\right),
\end{align*}
which implies $A_{x,y}\supseteq B_{x,y}$, thus scale-free percolation dominates long-range percolation in the sense that all the open edges in the LRP remain open in SFP. 
We therefore get that distances in LRP are an upper bound for distances in SFP and in particular get the upper bound \eqref{upperbound2}.

For the remaining regimes, there are many rigorous results about the graph distance $D(x,y)$ as $|x-y|\to \infty$. When $\alpha <d$, Benjamini, Kesten, Peres and Schramm \cite{Ben04} show that $D(x,y)$ is bounded by some (explicit) constant. When $\alpha> 2d$, Berger \cite{Ber04} shows that $D(x,y)$ $\geq |x-y|$. For the  borderline case $\alpha=2$ for $d=1$, Ding and Sly \cite{Din13} show that $D(x,y)\approx |x-y|^{\delta}$ for some $\delta\in (0,1)$.

\subsection{Related work} \label{sec-lit}
Various aspects of scale-free percolation (also known as \textit{heterogeneous long-range percolation}) have been investigated in the literature, both on the lattice $\Z^d$ \cite{Dep15, Hey17} as well as a continuum analogue \cite{Dal19, Dep19}, where vertices are given as a Poisson point process. The results in the present paper have been obtained on $\Z^d$, but it appears that we do not make use of the lattice structure in any crucial way, so that analogue results should hold for a continuum version of the model. 
A continuum version of scale-free percolation on finite domain (properly rescaled) is known a \emph{geometric inhomogeneous random graph} (GIRG), see \cite{Bri19, Bri20, Bri17}.

It has been pointed out recently \cite{Gra19, Gra20} that (continuum) scale-free percolation, as well as many other random graphs models, can be understood as special cases of the \emph{weight-dependent random connection models}. 
In the language of \cite{Gra19}, scale-free percolation corresponds to the weight-dependent random-connection model with \emph{product kernel} and polynomial profile function. Mind that the parametrization in \cite{Gra19} is different, see in particular \cite[Table 2]{Gra19}. 

For related recent work on spatial preferential attachment graphs we refer to  \cite{Hir20}.

\subsection{Overview}
We first prove the lower bound in Section \ref{sec2}. More precisely, we show that the probability that there exists a shorter path than the lower bound is negligible. We do this by estimating the path probability and counting the eligible paths using a so-called ``hierarchy'' argument. In Section \ref{sec3} we prove the upper bound with a double edge argument.

\section{Proof of lower Bound}\label{sec2}
In order to prove the lower bound, we derive variants of Biskup's arguments \cite{Bis04} in the setting of scale-free percolation. Similar to \cite{Bis04}, we split up the argument into 3 propositions. 
The key difference between SFP and LRP is that adjacent edges in the former model are only \emph{conditionally} independent. We resolve this by adjusting the definition of a \emph{hierarchy} (below) and combine it with estimates from \cite{Hof13} to break up the dependence structure.
\vspace{3mm}
\begin{definition}\label{hchy}
Given an integer $n\geq 1$ and distinct vertices $x,y\in \mathbb{Z}^d$, we say that the collection
\begin{equation}
\mathcal{H}_n(x,y)=\{(z_{\sigma}):\sigma \in \{0,1\}^k, k=1,2,\dots,n; z_{\sigma}\in \mathbb{Z}^d\}
\end{equation}
is a hierarchy of depth $n$ connecting $x$ and $y$ if
\begin{itemize}
\item[1.] $z_0=x$ and $z_1=y$;
\item[2.] $z_{\sigma 00}=z_{\sigma 0}$ and $z_{\sigma 11}=z_{\sigma 1}$ for all $k=0,1,\dots,n-2$ and all $\sigma \in \{0,1\}^k$;
\item[3.] For all $k=0,1,\dots,n-2$ and all $\sigma \in \{0,1\}^k$ such that $z_{\sigma 01} \neq z_{\sigma 10}$, the edge $\{z_{\sigma 01},z_{\sigma 10}\}$ is open;
\item[4.] Each edge $\{z_{\sigma 01},z_{\sigma 10}\}$ as specified in part 3 appears only once in $\mathcal{H}_n(x,y)$;
\item[5.] For $z_{\sigma_1},z_{\sigma_2}$ in $\mathcal{H}_n(x,y)$ with $k\in\{0,1,\dots,n-1\} $, $\sigma_1,\sigma_2\in \{0,1\}^{k+1}$ and $\sigma_1\neq \sigma_2$ we have that  $z_{\sigma_1}=z_{\sigma_2}$ if and only if there exists $\sigma \in \{0,1\}^k$ such that ${\sigma_1}={\sigma 0}$ and ${\sigma_2}={\sigma 1}$. In this case, we call the vertices $z_{\sigma_1}$ and $z_{\sigma_2}$ degenerate, otherwise non-degenerate.
\end{itemize}
The vertices $(z_{\sigma})$ are called sites of the hierarchy $\mathcal{H}_n(x,y)$.
\end{definition}

\begin{figure}[!htb]

\begin{tikzpicture}
\draw[thick] (-2,1.5) -- (2.5,2.5);
\draw[thick]  (-2.5,-1.5) -- (-1,0.5);
\draw [thick] (-1,0.5) -- (-2.5,1);
\draw [thick] (-1.5,-2.5) -- (-2.5,-1.5);
\draw [thick] (3,0.5) -- (1,-1);
\draw [thick] (1,1.5) -- (2,0.5);
\draw [thick] (1.5,-1.5) -- (3,-2);
\filldraw (-2.5,-2.5) circle (1pt) node[align=left,   below] {$x$=$z_0$};
\filldraw (3.5,-2.5) circle (1pt) node[align=left,   below] {$y$=$z_1$};
\filldraw (-2,1.5) circle (1pt) node[align=left,   above] {$z_{01}$};
\filldraw (2.5,2.5) circle (1pt) node[align=left,   above] {$z_{10}$};
\filldraw (-2.5,-1.5) circle (1pt) node[align=left,   left] {$z_{001}(=z_{0011})$};
\filldraw (-1.5,-2.5) circle (1pt) node[align=left,   right] {$z_{0001}$};
\filldraw (-1,0.5) circle (1pt) node[align=left,   right] {$z_{010}(=z_{0100})$};
\filldraw (-2.5,1) circle (1pt) node[align=left,   below] {$z_{0110}$};
\filldraw (1,-1) circle (1pt) node[align=left,   left] {$z_{110}$};
\filldraw (3,0.5) circle (1pt) node[align=left,   above] {$z_{101}$};
\filldraw (1,1.5) circle (1pt) node[align=left,   above] {$z_{1001}$};
\filldraw (2,0.5) circle (1pt) node[align=left,   below] {$z_{1010}$};
\filldraw (1.5,-1.5) circle (1pt) node[align=left,   above] {$z_{1101}$};
\filldraw (3,-2) circle (1pt) node[align=left,   right] {$z_{1110}$};
\node[] at (-0.5,0) 
    {\normalsize $=z_{0101}$};
\node[] at (-3,-2) 
    {\normalsize $=z_{0010}$};

\end{tikzpicture}

\caption{A hierarchy of depth $4$ with two degenerate sites $z_{001}$ and $z_{010}$}\label{hierarchy}

\end{figure}
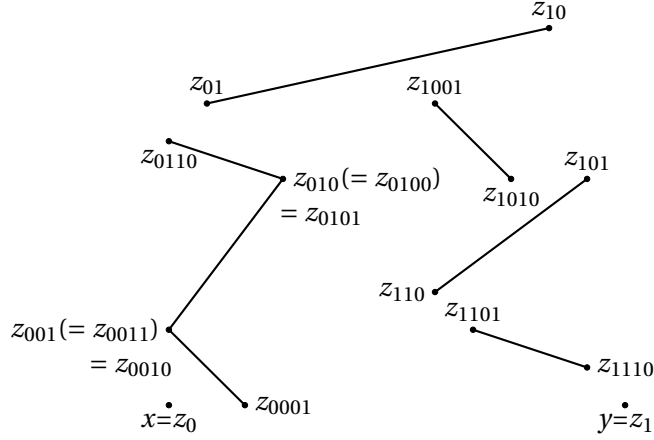

In the toy example depicted in Figure \ref{hierarchy}, the reader finds two overlapping sites. For $z_{001}(=z_{0011})$ and $z_{0010}$, there exists $\sigma=(0,0,1)\in \{0,1\}^3$ such that $z_{\sigma 1}=z_{\sigma0}$. Therefore, this is a degenerate site in the sense of Condition 5. Similarly  for $z_{010}$ and $z_{0101}$.

\begin{remark}
With only Conditions 1-4, our definition would coincide with Definition 2.1 in \cite{Bis04}. In addition, we impose Condition 5 to make sure that every element $\left(z_{\sigma}\right) \in \mathcal{H}_n(x,y)$ can be fitted into a vertex self-avoiding path connecting $x$ and $y$. By adding an additional condition, one realises the set of all hierarchies here is a subset of hierarchies defined in \cite{Bis04}, and this will be helpful when we count the eligible hierarchies, e.g.\ in \eqref{counting}.

The hierarchy $\mathcal{H}_n(x,y)$ is essentially a (random) subgraph of the complete graph with vertex set $\mathbb Z^d$. 
Condition 4 ensures that the number of open edges in this subgraph is at most $2^{n-1}$, and Condition 5 guarantees that the degree of all vertices in $\mathcal{H}_n$ is no more than 2.
\end{remark}

Since the shortest path connecting $x$ and $y$ is necessarily vertex self-avoiding, meaning that the weight of a single vertex appears at most twice in the path, we can estimate the probability of such a path by the Cauchy-Schwarz inequality. \\
\begin{lemma}[{\cite[Lemma 4.3]{Hof13}}]\label{singleedge}
Let $x,y\in \mathbb{Z}^d$ be distinct, then for all $\delta>0$, there exists a constant $C_{\delta}:=C(\delta, \lambda)>1$ such that 
\begin{equation}\label{edge}
\mathbb{E}\left[\left(\lambda\frac{W_xW_y}{|x-y|^{\alpha}}\wedge 1\right)^2\right]^{1/2} \leq C_{\delta}|x-y|^{-\alpha_1+\delta},
\end{equation}
where $\alpha_1$ is defined as in \eqref{alpha}.
\end{lemma}
\begin{proof}
From the proof of Lemma 4.3 in \cite{Hof13} we know
\begin{equation}
\mathbb{E}\left[\left(\lambda\frac{W_xW_y}{|x-y|^{\alpha}}\wedge 1\right)^2\right] \leq C_1\left(1+\log |x-y|\right)|x-y|^{-2\alpha_1}.
\end{equation}
for some constant $C_1\in (0,\infty)$. Then for all $\delta>0$, one has
\begin{equation*}
\lim_{r \to \infty}\frac{1+\log r}{r^{2\delta}}=0.
\end{equation*}
Hence there exists a constant $C_2>0$ such that $1+ \log r\leq C_2 r^{2\delta}$ for all $r >0$. Then we choose $C_{\delta}:=\sqrt{C_1C_2} \vee 2$ as desired.
\end{proof}
\begin{remark}
Actually, the estimation above can be further refined for $\tau>3$. If $\tau>3$, the weights $W_x$ and $W_y$ have finite variance. In this case, we can get rid of the $\delta$ in \eqref{edge}. On the other hand, since we can choose $\delta$ arbitrarily small, the refinement does not change our result. For our purpose, we choose $\delta$ small enough that $\alpha-\delta>d$ and $\alpha_1-\delta>d$.
\end{remark}

Now we estimate the probability that a path is open from above. Note that we call $\pi$ a path of length $n$ if there exist $n+1$ distinct vertices $x_0,\dots,x_n\in \mathbb{Z}^d$ such that $\pi=(x_0,\dots,x_n)$. We say that $\pi$ is open if all the edges $\{x_{i-1},x_i\}_{i=1,\dots,n}$ are open.

\begin{lemma}[{\cite[Thm.\ 4.2]{Hof13}}]\label{path}
Let $\pi:=(z_0,z_1,\dots,z_n) \in \left(\mathbb{Z}^d\right)^{n+1}$ be a path of length $n$. Then for all $\delta>0$, 
\begin{equation}
\mathbb{P}\left(\pi\text{ is open}\right)\leq \prod_{i=1}^{n}C_{\delta}|z_i-z_{i-1}|^{-\alpha_1+\delta},
\end{equation}
where the constant $C_{\delta}$ is as in Lemma \ref{singleedge}.
\end{lemma}
The proof of Lemma \ref{path} can be found in the proof of Theorem 4.2 in \cite{Hof13}, which combines the Cauchy-Schwarz inequality with the alternating independence of the edges in the path. With Lemma \ref{path}, one realises immediately that SFP behaves similarly to LRP in the sense that they have similar upper bounds for the probability of a path, which also indicates that the lower bound of SFP might be treated similar to LRP.

\begin{definition}
Let $x,y\in \mathbb{Z}^d$ be distinct, $\eta\in \big(0, \alpha_1/\left(2d\right)\big)$, and $n\geq 2$. We define $\mathcal{E}_n=\mathcal{E}_n(\eta)$ as the event that every hierarchy $\mathcal{H}_n(x,y)$ of depth $n$ connecting $x$ and $y$ such that
\begin{equation} \label{bond}
|z_{\sigma 01}-z_{\sigma 10}|\geq |z_{\sigma 0}-z_{\sigma 1}|(\log N)^{-\Delta_1} 
\end{equation}
holds for all $k=0,1,\dots,n-2$, and all $\sigma \in \{0,1\}^k$ also satisfy the bounds
\begin{equation} \label{gap}
\prod_{\sigma \in \{0,1\}^k}|z_{\sigma 0}-z_{\sigma 1}|\vee 1\geq N^{(2\eta)^k} \quad \text{ for all } k=1,2,\dots,n-1,
\end{equation}
where $N=|x-y|$ is the Euclidean distance between $x$ and $y$.
\end{definition}
With help of Lemma \ref{path} we now can estimate the probability of the event $\mathcal{E}_n$. 
\begin{proposition}\label{prop35}
Let $\eta\in \big(0, \alpha_1/(2d)\big)$.  Pick $\delta>0$ so small that $\alpha_1-\delta-d>0$ and $\alpha_1-\delta\in (2d\eta, \alpha_1)$, then there exists a constant $c_1>0$ such that for all $x,y\in \mathbb{Z}^d$ with $N=|x-y|$ satisfying $\eta^n\log N\geq 2(\alpha_1-\delta-d)$,
\begin{equation}
\mathbb{P}\left(\mathcal{E}^c_{n+1}\cap \mathcal{E}_n\right)\leq (\log N)^{c_1 2^n}N^{-(\alpha_1-\delta-2d\eta)(2\eta)^n},
\end{equation}
and
\begin{equation}
\P\left(\mathcal{E}_2^c\right)\leq \left(\log N\right)^{c_1}N^{-(\alpha_1-\delta-2d\eta)}.
\end{equation}
\end{proposition}
\begin{proof}
We modify the proof of Lemma 4.5 in \cite{Bis04} to fit our model.\\
Let $\mathcal{A}(n)$ be the set of all $2^n$-tuples $(z_{\sigma})$ of sites (or hierarchies) such that (\ref{bond}) holds for all $\sigma\in \big\{\{0,1\}^k:k=0,1,\dots,n-1\big\}$ and (\ref{gap}) is true  for $k=1,2,\dots,n-1$ but not for $k=n$. Then
\begin{equation}\label{en}
\mathbb{P}\left(\mathcal{E}^c_{n+1}\cap \mathcal{E}_n\right)\leq \sum_{(z_{\sigma})\in \mathcal{A}(n)}\mathbb{P}\left(\mathcal{H}_n(x,y) \text{ with sites }(z_{\sigma})\right).
\end{equation}
Here the event ``$\mathcal{H}_n(x,y) \text{ with sites }(z_{\sigma})$" means all the edges in this hierarchy with sites $(z_{\sigma})$ are open as in Condition 3 in Definition \ref{hchy}.

Now we fix one single hierarchy $\mathcal{H}_n(x,y)$ with sites $(z_{\sigma})$ and estimate its probability. Typically, a hierarchy consists of isolated edges, i.e., edges that do not share a common vertex. However, since we also allow degenerate vertices as in Condition 5 of Definition \ref{hchy}, there might be adjacent edges in the hierarchy. Nevertheless, we can decompose one hierarchy into several disjoint connected components, as exemplified in Figure \ref{hierarchy}. Condition 5 ensures that each of the connected components is an open path. 

\begin{example}
Consider the toy example in Figure \ref{hierarchy}. This hierarchy $\mathcal{H}_4(x,y)$ can be divided into 5 disjoint paths, namely
\begin{align*}
\pi_1=(z_{0110},z_{110},z_{001},z_{0001}),\quad \pi_2=&(z_{01},z_{10}),\\
 \pi_3=(z_{1001},z_{1010}), \quad \pi_4=(z_{101},z_{110}), \quad \pi_5=&(z_{1101},z_{1110}).
\end{align*}
\end{example}

Now assume that the hierarchy $\mathcal{H}_n(x,y)$ can be divided into $m$ disjoint open paths $\pi_i$, $ i=1,2,\dots,m$, with $\pi_i=(x_{i0},x_{i1},\dots,x_{im_i})$ and $x_{ij}\in (z_{\sigma})$. Then independence of edge occupation implies 
\begin{align*}
\mathbb{P}\left(\mathcal{H}_n(x,y) \text{ with sites }(z_{\sigma})\right)&= \mathbb{P}\left(\bigcap_{k=0}^{n-1}\bigcap_{\sigma\in \{0,1\}^k}\{z_{\sigma01}\sim z_{\sigma10}\}\right)\\
&=\mathbb{P}\left(\bigcap_{i=1}^m \{\pi_i \text{ is open}\}\right)=\prod_{i=1}^m\mathbb{P}\left(\pi_i\text{ is open}\right),
\end{align*}
where we rearrange the open edges in the hierarchy in the second step and use the fact that these open paths are vertex-disjoint and therefore independent in the last step. Further, 
\begin{align*}
\mathbb{P}\left(\mathcal{H}_n(x,y) \text{ with sites }(z_{\sigma})\right)&
\leq \prod_{i=1}^m \prod_{j=1}^{m_i}C_{\delta}\left|x_{im_j}-x_{im_{j-1}}\right|^{-\alpha_1+\delta}\\
&=\prod_{k=0}^{n-1} \prod_{\sigma \in \{0,1\}^k}\frac{C_{\delta}}{\left(\left|z_{\sigma01}-z_{\sigma10}\right|\vee 1\right)^{\alpha_1-\delta}},
\end{align*}
where we apply Lemma \ref{path} first and then bring the edges back in the original order again. In the last step we add the maximum with 1 to make sure that the denominator is not zero.\\
Likewise, we denote the ``gaps'' in the hierarchy by
\begin{align}
t_{\sigma}:=z_{\sigma 0}-z_{\sigma 1},
\end{align}
and $t_{\emptyset}:=x-y$.
With this notation, we rewrite condition (\ref{bond}) as 
\begin{align}\label{bond1}
|z_{\sigma 01}-z_{\sigma 10}|\geq |t_{\sigma}|(\log N)^{-\Delta_1}
\end{align}
and condition \eqref{gap} as
\begin{align}\label{gap1}
\prod_{\sigma \in \{0,1\}^k}|t_{\sigma}|\vee 1\geq N^{(2\eta)^k}.
\end{align}
Let $\mathcal{B}(k)$ be the set of all collections $(t_{\sigma})_{\sigma\in \{0,1\}^k}$ of vertices in $\mathbb{Z}^d$ such that \eqref{gap1} is true. Then \eqref{en} implies 
\begin{equation}
\mathbb{P}\left(\mathcal{E}^c_{n+1}\cap\mathcal{E}_{n}\right)\leq \left|\mathcal{B}^c(n)\right|\prod_{k=0}^{n-1}\left(\sum_{(t_{\sigma})\in \mathcal{B}(k)}\prod_{\sigma\in \{0,1\}^k}C_{\delta}\left(\frac{(\log N)^{\Delta_1}}{|t_{\sigma}|\vee 1}\right)^{\alpha_1-\delta}\right)
\end{equation}
Note that for $k=0$, we have $|t_{\emptyset}|=N$. 
Hence the estimation above can be written as
\begin{equation} \label{Bn}
\left|\mathcal{B}^c(n)\right|\frac{\left(C_{\delta}(\log N)^{\Delta_1(\alpha_1-\delta)}\right)^{2^n}}{N^{\alpha_1-\delta}}\prod_{k=1}^{n-1}\left(\sum_{(t_{\sigma})\in \mathcal{B}(k)}\prod_{\sigma\in \{0,1\}^k}\frac{C_{\delta}}{\left(|t_{\sigma}|\vee 1\right)^{\alpha_1-\delta}}\right),
\end{equation}

For each $k$ there are at most $2^k$ multipliers in the product over all $\sigma\in \{0,1\}^k$ (the number is smaller if there exist degenerate sites). Therefore, there are in total $\sum_{k=0}^{n-1}2^k=2^n-1$ and we get the exponent $2^n$ in the numerator in the first fraction.

In addition, for $n=2$, the event $\mathcal{E}^c_2$ means that there exists a hierarchy with sites $(z_{\sigma})$ of depth 2 such that
\begin{equation}
\left|z_{01}-z_{10}\right|\geq \left|z_0-z_1\right|\left(\log N\right)^{-\Delta_1}=N\left(\log N\right)^{-\Delta_1},
\end{equation}
and 
\begin{equation}
|z_0-z_{01}||z_{11}-z_1|\leq N^{2\eta}.
\end{equation}
Therefore
\begin{equation}\label{E2c}
\P(\mathcal{E}^c_2)\leq \sum_{(t_{\sigma})\notin \mathcal{B}(1)}\P(z_{01}\sim z_{10})\leq |\mathcal{B}^c(1)|\frac{C_{\delta}\left(\log N\right)^{\Delta_1(\alpha_1-\delta)}}{N^{\alpha_1-\delta}}.
\end{equation}

In order to estimate \eqref{Bn} and \eqref{E2c}, we need two lemmas from the appendix of \cite{Bis04}.
First for $\kappa\in \mathbb{N}$ and $b>0$, we let
\begin{equation}
\Theta_\kappa(b)=\bigg\{ (n_i)\in \mathbb{N}^\kappa\colon n_i\geq 1, \prod_{i=1}^\kappa n_i\geq b^\kappa\bigg\},
\end{equation}
and $\Theta^c_\kappa(b)$ be its complement in $\N^\kappa$. Then one has the following estimates.
\begin{itemize}
\item[(A1)][Lemma A.1 in \cite{Bis04}]\label{lemmaA1}
For each $\epsilon>0$ there exists a constant $g_1=g_1(\epsilon)<\infty$ such that
\begin{equation}\label{A1}
\sum_{(n_i)\in \Theta_{\kappa}(b)}\prod_{i=1}^\kappa\frac{1}{n_i^{1+\beta}}\leq (g_1b^{-\beta}\log b)^\kappa
\end{equation}
is true for all $\beta>0$, all $b>1$ and all $\kappa\in \N$ with
\begin{equation}
\beta-\frac{\kappa-1}{\kappa\log b}\geq \epsilon.
\end{equation}
\item[(A2)] [Lemma A.2 in \cite{Bis04}]\label{lemmaA2}
There exists a constant $g_2<\infty$ such that for each $\beta>1$, each $b\geq e/4$ and any $\kappa\in \N$,
\begin{equation}
\sum_{(n_i)\in \Theta^c_\kappa(b)}\prod_{i=1}^{\kappa}n_i^{\beta-1}\leq (g_2b^{\beta}\log b)^\kappa.
\end{equation}
\end{itemize}

Let $(n_{\sigma})$ be a collection of positive integers with $n_{\sigma}\leq  |t_{\sigma}|\vee 1<n_{\sigma}+1$. Note that $|\{x\in \mathbb{Z}^d\colon n\leq  |x|\vee 1<n+1\}|\leq cn^{d-1}$ for some positive constant $c=c(d)$ independent of $n$. 
Then for each $n_{\sigma}$ there exists at most $cn_{\sigma}^{d-1}$ such $t_{\sigma}$'s. Therefore,
\begin{align*}
\sum_{(t_{\sigma})\in \mathcal{B}(k)}\prod_{\sigma\in \{0,1\}^k}\frac{C_{\delta}}{\left(|t_{\sigma}|\vee 1\right)^{\alpha_1-\delta}}&\leq \sum_{(n_{\sigma})\in \Theta_{2^k}(N^{\eta^k})}\prod_{\sigma\in \{0,1\}^k}\left(cn_{\sigma}^{d-1}\frac{C_{\delta}}{n^{\alpha_1-\delta}_{\sigma}}\right)\\
&\leq\frac{(C_{\delta}cg_1)^{2k}(\eta^k)^{2^k}(\log N)^{2^k}}{N^{\eta^k2^k(\alpha_1-\delta-d)}},
\end{align*}
where we have applied (A1) in the last step with $\beta=\alpha_1-\delta-d,b=N^{\eta^k}$ and $\kappa=2^k$. Since $\eta<1$, we obtain the further bound
\begin{equation}
\sum_{(t_{\sigma})\in \mathcal{B}(k)}\prod_{\sigma\in \{0,1\}^k}\frac{C_{\delta}}{\left(|t_{\sigma}|\vee 1\right)^{\alpha_1-\delta}}\leq \frac{(C_1\log N)^{2^k}}{N^{(\alpha_1-\delta-d)(2\eta)^k}},
\end{equation}
where we choose $C_1:=cC_{\delta}g_1$.
Now it is left to estimate the size of $\mathcal{B}^c(n)$, and this can be done with help of (A2) as
\begin{equation}
\sum_{(t_{\sigma})\notin \mathcal{B}^c(n)} 1 \leq (C_2\log N)^{2^n}N^{d(2\eta)^n} 
\end{equation}
with $\beta=d,b=N^{\eta^n}$ and $\kappa=2^n$.

Now \eqref{Bn} can be simplified to
\begin{align}
&(C_2\log N)^{2^n}N^{d(2\eta)^n}\frac{\left(C_{\delta}(\log N)^{\Delta_1(\alpha_1-\delta)}\right)^{2^n}}{N^{\alpha_1-\delta}}\prod_{i=1}^{n-1}\frac{(C_1\log N)^{2^k}}{N^{(\alpha_1-\delta-d)(2\eta)^k}}.\nonumber \\
&\leq \left(C_1C_2C_{\delta}(\log N)^{\Delta_1(\alpha_1-\delta)+2}\right)^{2^n}N^{-\left((\alpha_1-\delta-d)\sum_{k=1}^{n-1}(2\eta)^k+\alpha_1-\delta-d(2\eta)^n\right)}\nonumber \\
&\leq (\log N)^{c_1 2^n}N^{-(\alpha_1-\delta-2d\eta)(2\eta)^n},\label{appofa1}
\end{align}
where the last step uses the bound 
\begin{equation}
(\alpha_1-\delta-d)\sum_{k=1}^{n-1}(2\eta)^k+\alpha_1-\delta-d(2\eta)^n\geq (\alpha_1-\delta-2d\eta)(2\eta)^n.
\end{equation}

\end{proof}

Our further strategy is to show that an open path with distance shorter then poly-logarithm is impossible. More precisely, we show that the existence of a shorter path is contained in some event with negligible probability. The event we use is as follows.\\
\begin{definition}\label{fn}
Let $x,y\in \mathbb{Z}^d$ be distinct and $n\in \mathbb{N}$. We define $\mathcal{F}_n:=\mathcal{F}_n(x,y)$ as the event that for every hierarchy of depth $n$ connecting $x$ and $y$ and satisfying \eqref{bond}, every collection of (vertex self-avoiding and) mutually disjoint paths $\pi_{\sigma}$ with $\sigma \in \{0,1\}^{n-1}$ such that $\pi_{\sigma}$ connects $z_{\sigma 0}$ and $z_{\sigma 1}$ without using any vertex from the hierarchy (except for the endpoints $z_{\sigma 0}$ and $z_{\sigma 1}$) obeys the bound
\begin{equation}\label{gapedges}
\sum_{\sigma \in \{0,1\}^{n-1}}|\pi_{\sigma}|\geq 2^n.
\end{equation}
\end{definition}

It might be instructive to look at the complement $\mathcal{F}_n^c$: this is the event that there exists such a hierarchy between $x$ and $y$ satisfying \eqref{bond}, but the edges filling the gaps violate \eqref{gapedges}. 
In the following proposition, we construct such a hierarchy in $\mathcal{F}_n^c$ from the shortest path.

\vspace{1.5mm}
\begin{proposition}[Lemma 4.6 in \cite{Bis04}]\label{prop37}
Let $\epsilon\in (0,\Delta_1)$. If $N=|x-y|$ is sufficiently large and 
\begin{equation}\label{boundforn}
n>\frac{\Delta_1-\epsilon}{\log 2}\log \log N,
\end{equation}
then
\begin{equation}
\{D(x,y)\leq (\log N)^{\Delta_1-\epsilon}\} \cap \mathcal{F}_n=\emptyset.
\end{equation}
\end{proposition}
\begin{proof}
The proof of Lemma 4.6 in \cite{Bis04} still holds here for the event with modified hierarchy, because the hierarchy there was constructed from the shortest path in which all the vertices have degree at most 2.
\end{proof}

\vspace{2mm}
Now we start to fill the "gaps" in the hierarchy. More precisely, we relate the events $\mathcal{E}_n$ and $\mathcal{F}_n$ by the following proposition.
\begin{proposition}\label{prop38}
Let $\eta\in (0,\alpha_1/(2d))$. For $\delta>0$ so small that $\alpha_1-\delta-d>0$ and $\alpha_1-\delta\in (2d\eta, \alpha_1)$, there exists a constant $c_2>0$ such that for all distinct $x,y\in \mathbb{Z}^d$ with $N=|x-y|$ satisfying $\eta^n\log N\geq 2(\alpha_1-\delta-d)$,
\begin{equation}
\mathbb{P}\left(\mathcal{F}^c_n\cap \mathcal{E}_n\right)\leq \left(\log N\right)^{c_22^n}N^{-(\alpha_1-\delta)(2\eta)^{n-1}}.
\end{equation}

\end{proposition}
\vspace{3mm}
The idea of proof is to first fix one hierarchy with the sites $(z_{\sigma})$, and estimate the probability that the paths that fill the gaps of this hierarchy have a certain length. Then the gap-filling paths and the open edges in the hierarchy constitute a path connecting $x$ and $y$. With help of Lemma \ref{path} we get the upper bound by summing over all possible hierarchies.

\begin{proof}
Let $\mathcal{A}^*(n)$ be the set of all collections $(z_{\sigma})$, $\sigma \in \{0,1\}^n$, satisfying \eqref{bond} for $k=0,1,\dots,n-2$ and \eqref{gap} for $k=1,2,\dots,n-1$. Then
\begin{equation}\label{fncen}
\mathbb{P}(\mathcal{F}_n^c \cap \mathcal{E}_n)=\sum_{(z_{\sigma})\in \mathcal{A}^*(n)}\mathbb{P}\left(\mathcal{F}^c_n\cap \mathcal{H}_n \text{ on } (z_{\sigma})\right).
\end{equation}
Here "$\mathcal{F}_n^c\cap\mathcal{H}_n$ on $(z_{\sigma})$" means that $ \mathcal{H}_n$ with sites $(z_{\sigma})$ is a hierarchy satisfying $\mathcal{F}_n^c$, as we have explained after Definition \ref{fn}.

We estimate the summands  on the right hand side of \eqref{fncen} by considering all possible lengths of $\pi_{\sigma}$. More precisely, let $(m_{\sigma})$ be a tuple of non-negative integers for $\sigma \in \{0,1\}^{n-1}$. Then

\begin{equation}\label{gaplength}
\mathbb{P}\left(\mathcal{F}^c_n\cap \mathcal{H}_n \text{ on } (z_{\sigma})\right)=\sum_{(m_{\sigma})}\mathbb{P}\left(\mathcal{F}^c_n\cap \mathcal{H}_n \text{ on } (z_{\sigma}) \text{ with } (|\pi_{\sigma}|)=(m_{\sigma})\right).
\end{equation}
Note that the open path $\pi_{\sigma}$ fills the gap between $z_{\sigma 0}$ and $z_{\sigma 1}$ in $\mathcal{H}_n$ for all $\sigma\in \{0,1\}^{n-1}$. All such open paths together with all the open edges $(z_{\sigma 01}, z_{\sigma 10}),\sigma\in \{0,1\}^{n-2}$, constitute a self-avoiding open path between $x$ and $y$. Let $\Gamma_{\sigma}(m_{\sigma})$ be the set of all path of length $m_{\sigma}$ connecting $z_{\sigma 0}$ and $z_{\sigma 1}$, that is, 
\[ \Gamma_{\sigma}(m_{\sigma})=\big\{\pi:\pi=(x_0,x_1,\dots,x_{m_{\sigma}}) \text{ with } x_0=z_{\sigma 0} \text{ and }x_{m_{\sigma }}=z_{\sigma1}\big\}.\] 
Now we estimate the probability in \eqref{gaplength} as
\begin{align}
&\mathbb{P}\left(\mathcal{F}^c_n\cap \mathcal{H}_n \text{ on } (z_{\sigma}) \text{ with } (|\pi_{\sigma}|)=(m_{\sigma})\right)\nonumber\\
=&\mathbb{E}\left[\mathbb{P}\left(\mathcal{F}^c_n\cap \mathcal{H}_n \text{ on } (z_{\sigma}) \text{ with } (|\pi_{\sigma}|)=(m_{\sigma})\right)|\left(W_x \right)_{x\in \mathbb{Z}^d}\right]\nonumber\\
=&\mathbb{E}\left[\mathbb{P}\left(\bigcap_{\sigma\in \{0,1\}^{n-1}}\{z_{\sigma 0}\stackrel []{\pi_{\sigma}}{\leftrightarrow} z_{\sigma 1}\}\bigcap_{\sigma\in \{0,1\}^{n-2}}\{z_{\sigma01}\sim z_{\sigma 10}\}\middle| \left(W_x\right)_{x\in \mathbb{Z}^d}\right)\right], \label{fullpath}
\end{align}
where $\{z_{\sigma 0}\stackrel []{\pi_{\sigma}}{\leftrightarrow} z_{\sigma 1}\}$ means $\pi_{\sigma}$ connects $z_{\sigma0}$ and $z_{\sigma 1}$.\\
By the conditional independence of edges, we rewrite \eqref{fullpath} as
\begin{align}
&\mathbb{P}\left(\mathcal{F}^c_n\cap\mathcal{H}_n \text{ on }(z_{\sigma}) \text{ with }(|\pi_{\sigma}|)=(m_{\sigma})\right) \nonumber\\
\leq &\sum_{\substack{(\pi_{\sigma}): \pi_{\sigma}=(x_{\sigma 0},\dots,x_{\sigma m_{\sigma}}) \\ \text{vertex-disjoint}}}\mathbb{E}\left[\prod_{\sigma \in \{0,1\}^{n-1}}\mathbb{P}\left(\pi_{\sigma}\middle|\left(W_x\right)_{x \in \mathbb{Z}^d} \right)\prod_{k=0}^{n-2}\prod_{\sigma' \in \{0,1\}^k}p_{z_{\sigma' 01}z_{\sigma' 10}}\right]\nonumber\\
=&\sum_{\substack{(\pi_{\sigma}): \pi_{\sigma}=(x_{\sigma 0},\dots,x_{\sigma m_{\sigma}}) \\ \text{vertex-disjoint}}}\mathbb{E}\left[\prod_{\sigma \in \{0,1\}^{n-1}}\prod_{i=1}^{m_{\sigma}}p_{x_{\sigma (i-1)},x_{\sigma i}}\prod_{k=0}^{n-2}\prod_{\sigma' \in \{0,1\}^k}p_{z_{\sigma' 01}z_{\sigma' 10}}\right] \label{twoparts}
\end{align}
where we sum over all possible paths between $z_{\sigma0}$ and $z_{\sigma1}$ for all $\sigma\in \{0,1\}^{n-1}$ and $p_{xy}$ is the connection probability as in \eqref{prob}.

In the expectation in \eqref{twoparts} the probability is divided into two parts: the first double product involves the edges filling the gaps in the hierarchy while the second double product is about the open edges in the hierarchy.\\

\begin{figure}[!h]

\centering
\begin{tikzpicture}

\draw[thick] (-3,0.5) -- (3,-1);
\draw[thick] (-4.5,-1.5) -- (-2.5,-1);
\draw[thick] (3.2,-0.5) -- (4,0);

\draw[dashed] (-5.5,-1) -- (-4.5,-1.5);
\draw[dashed] (4,0) -- (4.2,-0.8);
\draw[dashed] (3.2,-0.5) -- (3,-1);
\draw[dashed] (4.2,-0.8) -- (4.5,-1);
\draw[dashed] (-3,0.5) -- (-3.5,0);
\draw[dashed] (-3.5,-0) -- (-2.8,-0.2);
\draw[dashed] (-2.5,-1) -- (-2.8,-0.2);

\filldraw (-5.5,-1) circle (1pt) node[align=left,   above] {$x$=$z_0$};
\filldraw (4.5,-1) circle (1pt) node[align=left,   below] {$y$=$z_1$};
\filldraw (-3,0.5) circle (1pt) node[align=left,   above] {$z_{01}$};
\filldraw (3,-1) circle (1pt) node[align=left,   below] {$z_{10}$};

\filldraw (-4.5,-1.5) circle (1pt) node[align=left,   below] {$z_{001}$};
\filldraw (-2.5,-1) circle (1pt) node[align=left,   right] {$z_{010}$};
\filldraw (4,0) circle (1pt) node[align=left,   right] {$z_{110}$};
\filldraw (3.2,-0.5) circle (1pt) node[align=left,   left] {$z_{101}$};

\node[] at (-4.8,-1.1)  {\small $ \pi_{00}$};
\node[] at (-3,-0.5)  {\small $ \pi_{01}$};
\node[] at (3.4,-0.8)  {\small $ \pi_{10}$};
\node[] at (4.5,-0.7)  {\small $ \pi_{11}$};

\end{tikzpicture}
\caption{A hierarchy of depth 3 with site $(z_{\sigma})_{\sigma\in \{0,1\}^3}$. The gap-filling paths are $\{\pi_{\sigma}\}$ with $\sigma\in \{0,1\}^2$. In this example $|\pi_{00}|=1,|\pi_{01}|=3,|\pi_{10}|=1,|\pi_{11}|=2$, and $\sum|\pi_{\sigma}|=7<2^3=8$. We see that the paths here, together with the edges in the hierarchy, form a path connecting $x$ and $y$. }

\end{figure}
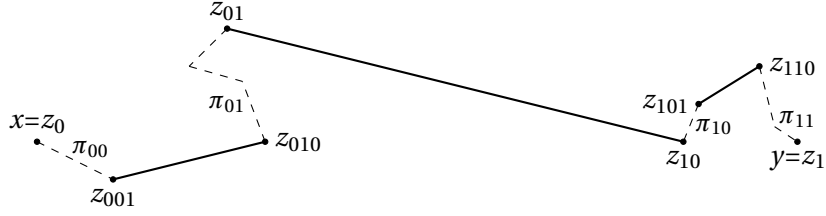

Note that all these paths $(\pi_{\sigma})$ have mutually disjoint vertices. Therefore, for fixed sites $(z_{\sigma})$ and fixed paths $(\pi_{\sigma})$, we obtain a self-avoiding open path starting from $x$ and ending in $y$. Now we use Lemma \ref{path} to bound the probability of this path, i.e. the expectation in \eqref{twoparts} as
\begin{align}\nonumber
&\mathbb{E}\left[\prod_{\sigma \in \{0,1\}^{n-1}}\prod_{i=1}^{m_{\sigma}}p_{x_{\sigma (i-1)},x_{\sigma i}}\prod_{k=0}^{n-2}\prod_{\sigma' \in \{0,1\}^k}p_{z_{\sigma' 01}z_{\sigma' 10}}\right]\\
\leq &\prod_{\sigma \in \{0,1\}^{n-1}}\prod_{i=1}^{m_{\sigma}}\frac{C_{\delta}}{\left(|x_{\sigma (i-1)}-x_{\sigma i}|\vee 1\right)^{\alpha_1-\delta}}\prod_{k=0}^{n-2}\prod_{\sigma' \in \{0,1\}^{k}}\frac{C_{\delta}}{|z_{\sigma' 01}-z_{\sigma' 10}|^{\alpha_1-\delta}}.
\end{align}
Then \eqref{twoparts} becomes
\begin{align}\nonumber
&\mathbb{P}\left(\mathcal{F}^c_n\cap\mathcal{H}_n \text{ on }(z_{\sigma}) \text{ with }(|\pi_{\sigma}|)=(m_{\sigma})\right)\\
\leq &\sum_{(\pi_{\sigma})}\prod_{\sigma \in \{0,1\}^{n-1}}\prod_{i=1}^{m_{\sigma}}\frac{C_{\delta}}{\left(|x_{\sigma (i-1)}-x_{\sigma i}|\vee 1\right)^{\alpha_1-\delta}}\prod_{k=0}^{n-2}\prod_{\sigma' \in \{0,1\}^{k}}\frac{C_{\delta}}{|z_{\sigma' 01}-z_{\sigma' 10}|^{\alpha_1-\delta}}\nonumber\\
= &\left(\prod_{\sigma \in \{0,1\}^{n-1}}Q_{m_{\sigma}}(z_{\sigma 0},z_{\sigma 1})\right)\prod_{k=0}^{n-2}\prod_{\sigma' \in \{0,1\}^{k}}\frac{C_{\delta}}{|z_{\sigma' 01}-z_{\sigma' 10}|^{\alpha_1-\delta}}
\end{align}
where
\begin{equation}
Q_m(u,v):=\sum_{\substack{\pi=(x_0,\dots,x_m)\\x_0=u, \text{ },x_m=v}}\prod_{i=1}^{m}\frac{C_{\delta}}{\left(|x_{i-1}-x_i|\vee 1\right)^{\alpha_1-\delta}}.
\end{equation}
Here the sum runs over self-avoiding paths $\pi$ of length $m$, and therefore $Q_m(u,v)$ is the upper bound for the probability that $u$ and $w$ are connected by an open path with length $m$. Note the fact that for all $u,v\in \mathbb{Z}^d$ with $u \neq v$ and $\alpha>d$, there exits a constant $a \in (0,\infty)$, independent of $u$ and $v$, such that
\begin{equation}\label{fact}
\sum_{w\in \mathbb{Z}^d, w\notin \{u,v\}}\frac{1}{|u-w|^{\alpha}}\frac{1}{|v-w|^{\alpha}}\leq \frac{a}{|u-v| ^{\alpha}}.
\end{equation}
The estimate above can be obtained by splitting the sum in two cases: $\{w\in \mathbb{Z}^d: |u-w|\leq |v-w|\}$ and $\{w\in \mathbb{Z}^d: |u-w|> |v-w|\}$. In the first case one has $|v-w|\geq 1/2|u-v|$, and since $\alpha>d$, $\sum_{w\neq u} 1/|u-w|^{\alpha}<\infty$. A similar argument holds for the second case.

Then we can bound $Q_m(u,v)$ from above by applying \eqref{fact} $m$ times iteratively and obtain
\begin{equation}\label{Q}
Q_m(u,v)\leq \frac{(C_{\delta}a)^m}{\left(|u-v|\vee 1\right)^{\alpha_1-\delta}}.
\end{equation}
If we now sum over all the possible combinations of $(m_{\sigma})$ with $\sum_{\sigma}m_{\sigma}<2^n$, we obtain the upper bound
\begin{align}
&\mathbb{P}\left(\mathcal{F}_n^c\cap \mathcal{H}_n \text{ on } (z_{\sigma})\right) \nonumber\\
\leq &\sum_{(m_{\sigma}): \sum_{\sigma}m_{\sigma}<2^n}\left(\prod_{\sigma \in \{0,1\}^{n-1}}Q_{m_{\sigma}}(z_{\sigma 0},z_{\sigma 1})\right)\prod_{k=0}^{n-2}\prod_{\sigma' \in \{0,1\}^{k}}\frac{C_{\delta}}{|z_{\sigma' 01}-z_{\sigma' 10}|^{\alpha_1-\delta}}\nonumber\\
\leq &(4aC_{\delta})^{2^n}\prod_{\sigma\in \{0,1\}^{n-1}}\frac{1}{\left(|z_{\sigma 0}-z_{\sigma 1}|\vee 1\right)^{\alpha_1-\delta}}\prod_{k=0}^{n-2}\prod_{\sigma' \in \{0,1\}^{k}}\frac{C_{\delta}}{|z_{\sigma' 01}-z_{\sigma' 10}|^{\alpha_1-\delta}}\nonumber\\
\leq &(4aC_{\delta})^{2^n} \prod_{k=0}^{n-1}\prod_{\sigma \in \{0,1\}^k}\frac{C_{\delta}(\log N)^{(\alpha_1-\delta)\Delta'}}{\left(|z_{\sigma 0}-z_{\sigma 1}|\vee 1\right)^{\alpha_1-\delta}}.\label{numberofcom}
\end{align}
Here we first used the estimation for $Q_m(u,v)$ in \eqref{Q} and the fact that the number of such eligible tuples $(m_{\sigma})$ is at most $4^{2^n}$, and subsequently used the fact that on $\mathcal{E}_n$ the lengths of open edges in the hierarchy are subject to the constrain \eqref{bond1}.

\medskip
We now can estimate the desired probability as
\begin{align}
\mathbb{P}(\mathcal{F}^c_n\cap \mathcal{E}_n)=&\sum_{(z_{\sigma})\in \mathcal{A}^*(n)}(4aC_{\delta})^{2^n} \prod_{k=0}^{n-1}\prod_{\sigma \in \{0,1\}^k}\frac{C_{\delta}(\log N)^{(\alpha_1-\delta)\Delta'}}{\left(|z_{\sigma 0}-z_{\sigma 1}|\vee 1\right)^{\alpha_1-\delta}}\label{counting}\\
&\leq \frac{\left(C_1(\log N)^{\Delta_1(\alpha_1-\delta)}\right)^{2^n}}{N^{\alpha_1-\delta}}\prod_{k=0}^{n-1}\sum_{(t_{\sigma})\in \mathcal{B}(k)}\prod_{\sigma\in \{0,1\}^k}\frac{C_{\delta}}{\left(|t_{\sigma}|\vee 1\right)^{\alpha_1-\delta}}.
\end{align}
Recall that $\mathcal{B}(k)$ is the set of all collections $(t_{\sigma}),\sigma\in \{0,1\}^k$, of vertices in $\mathbb{Z}^d$ such that \eqref{gap1} is true. Then by applying \eqref{A1}
again (as in \eqref{appofa1}), together with
\begin{equation}
\alpha_1-\delta+(\alpha_1-\delta)\sum_{k=1}^{n-1}(2\eta)^k\geq (\alpha_1-\delta)(2\eta)^{n-1},
\end{equation}
the result follows.
\end{proof}

\bigskip

\begin{proof}[Proof of Theorem \ref{thm21}, lower bound]
By Proposition \ref{prop37} we can bound the probability of the event $\{D(x,y)\leq (\log N)^{\Delta_1-\epsilon}\}$ by the probability of the event  $\mathcal{F}_n^c$ once \ref{prop37} holds. That is, if the depth of the hierarchy $n$ satisfies (\ref{boundforn}),
\begin{align}
&\P\left(D(x,y)\leq (\log N)^{\Delta_1-\epsilon}\right)
\leq \P\left(\mathcal{F}^c_n\right).\label{twoprob}
\end{align}

Now we fix $\epsilon\in(0,\Delta_1-1)$. Since $2^{-\nicefrac1{\Delta_1}}=\nicefrac{\alpha_1}{2d}$ by \eqref{delta}, we can choose $\delta>0$ and $\eta$ such that 
\begin{equation}\label{eta}
2^{-\nicefrac{1}{(\Delta_1-\epsilon})}<\eta<\frac{\alpha_1-\delta}{2d},
\end{equation}
so that, in particular, 
 $\frac{\Delta_1-\epsilon}{\log 2}<\frac{1}{\log 1/\eta}$. 
We further fix $\delta_1\in (0, \alpha_1-\delta-2d\eta)$. 
For large $N$ we thus find $n\in\N$ such that 
\begin{equation} \label{nforsum2}
	\frac{\Delta_1-\epsilon}{\log 2}\log \log N
	<n
	\leq \frac{\log \log N+\log\frac{\delta_1}{c_1}-\log\log\log N}{\log 1/\eta}.
\end{equation}
We henceforth assume that $N$ is large enough that (for $c_1$ from Proposition \ref{prop35})
\begin{equation}\label{nforsum3a}
(\log N)^{c_12^n}\leq N^{\delta_1(2\eta)^n}.
\end{equation}
In this case, the right hand side of \eqref{nforsum2} is further bounded from above by 
\begin{equation}\label{nforsum3}
	\frac{\log\log N-\log2(\alpha_1-\delta-d)}{\log1 /\eta}. 
\end{equation}
Therefore, we may apply the assertions of Propositions \ref{prop35}, \ref{prop37} and \ref{prop38} (Proposition \ref{prop35} even for all smaller values of $n$), and we thus get  
\begin{align}\nonumber
\P\left(D(x,y)\leq (\log N)^{\Delta_1-\epsilon}\right)
&\leq \P\left(\mathcal{F}^c_n\right)\leq \P\left(\mathcal{E}^c_n\right)+\P\left(\mathcal{F}_n^c\cap\mathcal{E}_n\right)\\
&\leq \sum_{k=3}^n\P(\mathcal{E}^c_k\cap\mathcal{E}_{k-1})+\P(\mathcal{E}_2^c)+\P\left(\mathcal{F}_n^c\cap\mathcal{E}_n\right).\label{eqDlowerBd} 
\end{align}
Using Proposition \ref{prop35} and \eqref{nforsum3a}, we get for $k\leq n$ that 
\begin{equation}
\P(\mathcal{E}^c_{k+1}\cap\mathcal{E}_k)\leq N^{-(\alpha_1-\delta-2d\eta-\delta_1)(2\eta)^k}, 
\end{equation}
and Proposition \ref{prop38} yields a similar bound for $\P\left(\mathcal{F}_n^c\cap\mathcal{E}_n\right)$. 
Since $2\eta>1$, we thus get the right hand side of \eqref{eqDlowerBd} arbitrarily close to 0 by choosing $N$ sufficiently large. 

\color{black}

Translation invariance and the FKG-inequality yield
\begin{equation}
\P(x,y\in \mathcal{C}_{\infty})\geq \P(x\in \mathcal{C}_{\infty})^2>0.
\end{equation}
Therefore, we have
\begin{equation*}
\lim_{|x-y|\to \infty}\P\left(D(x,y)\leq \left(\log |x-y|\right)^{\Delta_1-\epsilon}\bigg|x,y\in \mathcal{C}_{\infty}\right)=0,
\end{equation*}
as desired.
\end{proof}

\section{Proof of upper bound}\label{sec3}

The upper bound in (b) and (c) of Theorem \ref{thm21} is already established in \cite{Dep15}, so that we can restrict our attention here to the case $\tau\in (2,3)$. Interestingly, for $\tau>3$, the logarithmic power of upper and lower bound match, and we thus identified the correct exponent.

Unlike in long-range percolation, edges in scale-free percolation are only conditionally independent. Intuitively speaking, adjacent edges are positively correlated due to the weight of their joint vertex (see Exercise 9.40 in Chapter 9 of \cite{Hof19}). Here we state a more general result, which is implied by the FKG-Inequality.

\begin{proposition}\label{pro31}
Let $\pi=(x_i)_{i=0,\dots,n}$ be a path in scale-free percolation and $k\in\{1,\dots,n-1\}$, and let $\pi_1,\pi_2$ be two subpaths of $\pi$ by cutting $\pi$ at vertex $x_k$. That is, $\pi_1=(x_i)_{i=0,\dots,k}$ and $\pi_2=(x_i)_{i=k+1,\dots,n}$. Then
\begin{equation}
\mathbb{P}\left(\pi \text{ is open}\right)\geq \mathbb{P}\left(\pi_1\text{ is open}\right)\mathbb{P}\left(\pi_2\text{ is open}\right).
\end{equation}
\end{proposition}

From Proposition \ref{pro31} we see that two adjacent edges (or even paths) in scale-free percolation are indeed positively correlated. The next result tells us that in some cases the positive correlation is significant.

\begin{proposition}[Probability of adjacent edges]\label{adjedge}
In scale-free percolation with $\tau \in (2,3)$ there exist $x_0>0$ and $c_2>c_1>0$ such that for all $x,y$ and $z\in \Z^d$ with $|x-y|\geq |y-z|\geq x_0$, we have
\begin{equation}\label{adjacentedges}
c_1|x-y|^{-\alpha}|y-z|^{-\alpha(\tau-2)}\leq \P(x\sim y\sim z)\leq c_2|x-y|^{-\alpha}|y-z|^{-\alpha(\tau-2)}.
\end{equation}
\end{proposition}

\begin{proof}
We start by calculating the probability of this joint occurrence as
\begin{equation}
\P(x\sim y\sim z)=\mathbb{E}\left[\left(1-\exp\left(-\frac{\lambda W_xW_y}{|x-y|^{\alpha}}\right) \right)\left(1-\exp\left(-\frac{\lambda W_yW_z}{|y-z|^{\alpha}}\right) \right)\right].
\end{equation}
For $t\in(0,\infty)$, one has
\begin{equation}\label{inequality}
\frac{1}{2}(t\wedge 1)\leq 1-e^{-t}\leq t\wedge 1,
\end{equation}
so that 
\begin{equation}\label{35}
\P(x\sim y\sim z) \leq \mathbb{E}\left[\left(\frac{\lambda W_xW_y}{|x-y|^{\alpha}}\wedge 1\right)\left(\frac{\lambda W_yW_z}{|y-z|^{\alpha}} \wedge 1\right)\right]\leq 4\,\P(x\sim y\sim z),
\end{equation}
and it is sufficient to compute the expectation in the middle. 

First we show that the two single weights $W_x$ and $W_z$ do not play a role in the result. 
On the one hand, we know $W_x\geq 1$, therefore
\begin{equation}
\mathbb{E}\left[\left(\frac{\lambda W_xW_y}{|x-y|^{\alpha}}\wedge 1\right)\left(\frac{\lambda W_yW_z}{|y-z|^{\alpha}} \wedge 1\right)\right]\geq \mathbb{E}\left[\left(\frac{\lambda W_y}{|x-y|^{\alpha}}\wedge 1\right)\left(\frac{\lambda W_y}{|y-z|^{\alpha}} \wedge 1\right)\right]
\end{equation}
One the other hand, 
\begin{equation}
st\wedge 1 \leq s(t\wedge 1) \qquad(s\geq 1 \text{ and }t > 0)
\end{equation}
implies 
\begin{align}
\mathbb{E}\left[\left(\frac{\lambda W_xW_y}{|x-y|^{\alpha}}\wedge 1\right)\left(\frac{\lambda W_yW_z}{|y-z|^{\alpha}} \wedge 1\right)\right] &\leq \mathbb{E}\left[W_x\left(\frac{\lambda W_y}{|x-y|^{\alpha}}\wedge 1\right)\left(\frac{\lambda W_y}{|y-z|^{\alpha}} \wedge 1\right)W_z\right] \nonumber \\
&={\mu}^2\mathbb{E}\left[\left(\frac{\lambda W_y}{|x-y|^{\alpha}}\wedge 1\right)\left(\frac{\lambda W_y}{|y-z|^{\alpha}} \wedge 1\right)\right],
\end{align}
where $\mu:=\mathbb{E}[W_x]< \infty$ since $\tau >2$.
Together with \eqref{35}, we thus obtain
\begin{equation}
\frac{1}{\mu^2}\mathbb{P}(x\sim y \sim z) \leq  \mathbb{E}\left[\left(\frac{\lambda W_y}{|x-y|^{\alpha}}\wedge 1\right)\left(\frac{\lambda W_y}{|y-z|^{\alpha}} \wedge 1\right)\right]\leq 4\P(x\sim y\sim z).
\end{equation}
Thus it suffices to compute the expectation
\begin{align*}
&\mathbb{E}\left[\left(\frac{\lambda W_y}{|x-y|^{\alpha}}\wedge 1\right)\left(\frac{\lambda W_y}{|y-z|^{\alpha}} \wedge 1\right)\right]\\
=&\int_{\mathbb{R}} \left(\frac{\lambda u}{|x-y|^{\alpha}}\wedge 1\right)\left(\frac{\lambda u }{|y-z|^{\alpha}} \wedge 1\right) dW_y(u)\\
=&\int_{1}^{\infty}\left(\frac{\lambda u}{|x-y|^{\alpha}}\wedge 1\right)\left(\frac{\lambda u }{|y-z|^{\alpha}} \wedge 1\right)(\tau-1)u^{-\tau}du
\end{align*}
We now split the domain of integration into three intervals: $[1,|y-z|^{\alpha}/\lambda]$, $|y-z|^{\alpha}/\lambda,|x-y|^{\alpha}/\lambda]$ and $(|x-y|^{\alpha}/\lambda, \infty)$. 
After some calculation, one obtains
\begin{align*}
&\mathbb{E}\left[\left(\frac{\lambda W_y}{|x-y|^{\alpha}}\wedge 1\right)\left(\frac{\lambda W_y}{|y-z|^{\alpha}} \wedge 1\right)\right]\\
=&\frac{\tau-1}{(3-\tau)(\tau-2)}\frac{\lambda}{|x-y|^{\alpha}}\frac{\lambda^{\tau-2}}{|y-z|^{\alpha(\tau-2)}}
-\frac{\tau-1}{3-\tau}\frac{\lambda}{|x-y|^{\alpha}}\frac{\lambda}{|y-z|^{\alpha}}-\frac{\lambda^{\tau-1}}{\tau-2}\frac{1}{|x-y|^{\alpha(\tau-1)}}.
\end{align*}
We thus may choose $c_2:=\frac{\tau-1}{(3-\tau)(\tau-2)}\mu^2\lambda^{\tau-1}$. 
For $\tau\in (2,3)$, we find that the first term dominates the sum when $|y-z|\rightarrow\infty$ (the other terms are negative, but the total sum is trivially nonnegative). 
Hence there exist positive constant $x_0$ and $c_1$ such that 
\begin{equation}
\P(x\sim y\sim z)\geq c_1|x-y|^{-\alpha}|y-z|^{-\alpha(\tau-2)} \quad \text{for }|y-z|\geq x_0.
\end{equation}
\end{proof}

In fact, the weights of two end points do not contribute to the significant positive correlation in Proposition \ref{adjedge}, as we formulate in the next corollary. 

\begin{corollary}\label{cor314}
In scale-free percolation with $\tau \in (2,3)$, there exist constants $c_i=c_i(a,b)>0$ for $i=1,2$ and $x_0=x_0(a,b)>0$ such that for all $x,y$ and $z\in \Z^d$ with $|x-y|\geq |y-z|\geq x_0$ we have 
\begin{equation*}
c_1|x-y|^{-\alpha}|y-z|^{-\alpha(\tau-2)}\leq \P\left(x\sim y\sim z|W_x=a,W_z=b\right)\leq c_2|x-y|^{-\alpha}|y-z|^{-\alpha(\tau-2)}.
\end{equation*}
In particular, for constants $M>m>0$, there exist $C_i=C_i(a,b, m, M)>0,i=1,2$ and $x_0'=x_0'(a,b)>0$ such that if $|x-y|$ and $|y-z|$ are comparable in the sense 
\begin{equation*}
m|x-y|\leq |y-z| \leq M|x-y|,
\end{equation*}
then
\begin{align*}
C_1|x-y|^{-\alpha(\tau-1)/2}|y-z|^{-\alpha(\tau-1)/2}\leq \P\left( x\sim y\sim z|W_x=a,W_z=b\right)\\
\leq C_2|x-y|^{-\alpha(\tau-1)/2}|y-z|^{-\alpha(\tau-1)/2},
\end{align*}
for all $|x-y|\geq x_0'$.
\end{corollary}

In light of Propositions \ref{pro31} and \ref{adjedge} and Corollary \ref{cor314}, we now aim to construct a path with edges of comparable length. Instead of connecting two vertices directly, we use an intermediate vertex as a ``bridge'' to connect the two vertices.
For $x,y\in\Z^d$, $A\subset\Z^d$, we write 
\[ \{x\sim A \sim y\}=\bigcup_{z\in A}\{x\sim z\sim y\}\]
for the event that $x$ and $y$ are connected via an ``intermediate vertex'' in $A$.

\begin{proposition}\label{pro315}
For $\beta\in (0,1)$,  there exist constants $N_0,K>0$ such that for all $x,y\in \mathbb{Z}^d$ with  $N:=|x-y|\geq N_0$ it is true that
\begin{equation}
\P\left(x\sim A \sim y \right) \geq  \frac{K}{N^{2\alpha_1-d\beta}},
\end{equation}
where 
\[A:=\Big(\frac12\big(x+y\big)+\big[-N^\beta,N^\beta\big]^d\Big)\cap\Z^d\]  
is the cube with side length $N^{\beta}$ centred at the middle point of the line segment between $x$ and $y$.
\end{proposition} 

\begin{proof}
Since $\beta<1$, there exist constants $l=l(\beta,d)$ and $L=L(\beta,d)$ with $L>l>0$ and $N_1>0$ such that 
\begin{align*}
lN\leq|x-z|\leq LN\quad \text{ and}\quad 
lN\leq|y-z|\leq LN,
\end{align*}
for all $z\in A$ and all $N\geq N_1$. Therefore, $|x-z|$ and $|y-z|$ are comparable in the sense of Proposition \ref{cor314}. Thus we have
\begin{align}
\P\left(x\sim A \sim y \right)
\,&\geq \P\left(x\sim A \sim y \middle|W_x=1,W_y=1\right)\nonumber\\
&=1-\prod_{z\in A}\left(1-\P\left(x\sim z \sim y\middle|W_x=W_y=1\right)\right),
\end{align}
where we used the conditional independence of edges and the independence of vertex weights. 
We estimate this further using Corollary \ref{cor314} and get that there exists $N_2>0$, $c_1>0$, such that for all $N\geq N_2$,
\begin{equation}\label{conprob}
\P\left(x\sim z \sim y\middle|W_x=W_y=1\right)\geq c_1\frac{1}{|x-z|^{\alpha_1}}\frac{1}{|z-y|^{\alpha_1}}\geq \frac{c_1}{L^{2\alpha_1}}\frac{1}{N^{2\alpha_1}}.
\end{equation}
Note that the right hand side of \eqref{conprob} is independent of $z$, which allows us to estimate
\begin{align*}
\P\left(x\sim A \sim y \right)\geq 1-\left(1-\frac{c_1}{L^{2\alpha_1}N^{2\alpha_1}}\right)^{N^{d\beta}}.
\end{align*}
Now we use the elementary bound 
\begin{equation}\label{inequality2}
1-t\leq e^{-t} \leq 1-\frac{1}{2}t \qquad (0<t<1)
\end{equation}
to conclude that
\begin{equation}
\left(1-\frac{c_1}{L^{2\alpha_1}N^{2\alpha_1}}\right)^{N^{d\beta}}\leq e^{-CN^{d\beta-2\alpha_1}/L^{2\alpha_1}}\leq 1-\frac{c_1}{2N^{2\alpha_1-d\beta}L^{2\alpha_1}}.
\end{equation}
Since $d\beta-2\alpha_1<0$, there exists $N_3>0$ such that we have $c_1N^{d\beta-2\alpha_1}/L^{2\alpha_1}<1$ for all $N\geq N_3$, and consequently also $c_1N^{-2\alpha_1}L^{-2\alpha_1}<1$.
Finally, we have 
\begin{equation}
\P(x\sim A \sim y)\geq \frac{c_1}{2N^{2\alpha_1-d\beta}L^{2\alpha_1}}
\end{equation}
for all $N\geq  N_0:=\max\{N_1,N_2,N_3\}$ and choose $K:=\frac{c_1}{2L^{2\alpha_1}}$ as desired.
\end{proof}

With these preparations we finally prove the upper bound.
\begin{proof}[Proof of Theorem \ref{thm21}, upper bound]
Since the adjacent paths in scale-free percolation are positively correlated (by Proposition \ref{pro31}) and the probability of the compound edge "$x\sim A \sim y$" decays algebraically with exponent $2\alpha_1-d\beta$ (by Proposition \ref{pro315}), we have that the probability of a path being open in SFP dominates that in LRP with edge probability decaying with exponent $2\alpha_1-d\beta$ in \eqref{asymlrp}. 
Therefore, the graph distance in SFP in this case is no more than twice the distance in long-range percolation with connection probability as in \eqref{asymlrp} but with $\alpha$ replaced by $2\alpha_1-d\beta$. Since one can choose $\beta$ arbitrarily close to 1, the result follows from Theorem \ref{lrp}. 
\end{proof}

\begin{remark}\label{remark}
In this paper, we made a specific choice for the connection probability in \eqref{prob}. However, the  proofs for both lower and upper bounds in Section \ref{sec2} and Section \ref{sec3} only require  asymptotics of the connection probability to estimate the path, for example in Lemma \ref{singleedge} and Proposition \ref{adjedge}. Therefore, our results generalise to the scale-free percolation with connection probability $p_{x,y}=\Theta\left(\frac{\lambda W_xW_y}{|x-y|^{\alpha}}\wedge 1\right)$, provided that a unique infinite cluster exists.
\end{remark}

\color{black}

\paragraph{\bf Acknowledgement.} We acknowledge support from Deutsche Forschungsgemeinschaft (DFG) under project 386248531. We thank Matthew Dickson for advise on the presentation. 

\bibliographystyle{amsplain}

\providecommand{\bysame}{\leavevmode\hbox to3em{\hrulefill}\thinspace}
\providecommand{\MR}{\relax\ifhmode\unskip\space\fi MR }
\providecommand{\MRhref}[2]{%
  \href{http://www.ams.org/mathscinet-getitem?mr=#1}{#2}
}
\providecommand{\href}[2]{#2}

\end{document}